\documentclass [11pt,oneside]{amsart}
\usepackage{amsmath}

\usepackage{setspace}
\usepackage {amsfonts}
\usepackage {amsmath}
\usepackage {amsthm}
\usepackage {amssymb}
\usepackage {framed}
\usepackage {amsxtra}
\usepackage {enumerate}
\usepackage {graphicx}
\usepackage{color}
\usepackage{graphicx}
\usepackage{graphicx}
\usepackage{wrapfig}
\usepackage{lscape}
\usepackage{rotating}
\usepackage{epstopdf}
\usepackage{pdflscape}
\usepackage{setspace}
\usepackage[left=0.86in, right=.86 in, top=1.2 in, footskip=0.5 in, bottom=0.9 in]{geometry}
\usepackage{adjustbox}

\usepackage{algorithm}
\usepackage{algpseudocode}
\usepackage{tikz}
\usepackage{hyperref}

\makeatletter

\theoremstyle{definition}
\newtheorem{df}{Definition} [section]

\theoremstyle{plain}

\newtheorem{claim}[df]{Claim}

\title{The Chromatic Number of the Plane is at least $5$ - a new proof. }
\author{Geoffrey Exoo$^1$}
\thanks{\mbox{$^1$Department of Mathematics and Computer Science, Indiana State University, Terre Haute, IN 47809, \texttt{ge@cs.indstate.edu}}}

\author{Dan Ismailescu$^2$}
\thanks{$^2$Mathematics Department, Hofstra University, Hempstead, NY 11549, \texttt{dan.p.ismailescu@hofstra.edu}}

\begin{document}


\begin{abstract}
We present an alternate proof of the fact that given any $4$-coloring of the plane there exist two points unit distance apart which are identically colored.
\end{abstract}
\maketitle

\section{\bf Introduction}

In 1950, Edward Nelson raised the problem of determining the minimum number of colors that are needed to color the points of the plane so that no two points unit distance apart are assigned the same color. This number is referred to as the {\it chromatic number of the plane}, and is denoted by $\chi(\mathbb{E}^2)$.

It is easy to show that $\chi(\mathbb{E}^2)\ge 4$ as Leo and William Moser \cite{MM}, and soon after, Golomb (cf. Soifer \cite{soifer}, p. 19) , constructed unit-distance graphs that require four colors - see figure \ref{mosergolomb}. An upper bound of $\chi(\mathbb{E}^2)\le 7$ follows from a hexagonal tiling argument due to Hadwiger \cite{hadwiger}.

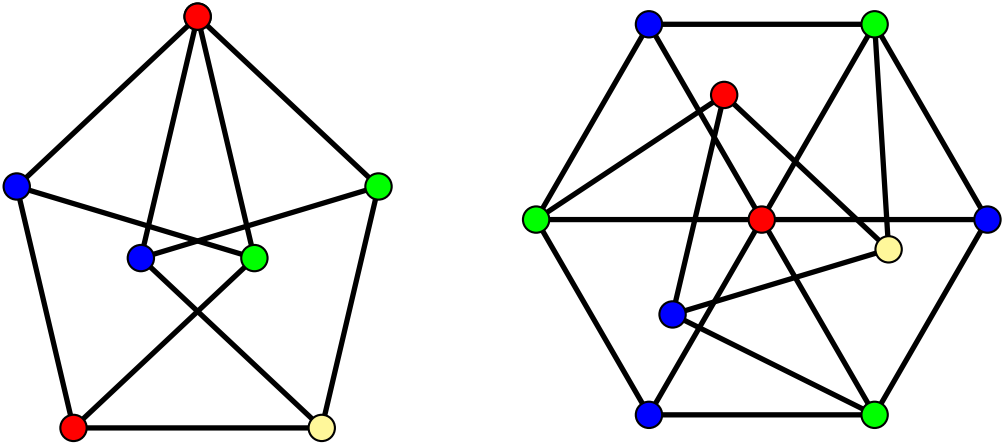
\begin{figure}[ht]
\centering

\begin{tikzpicture}[scale=3.0,line width=2.0pt]

\tikzstyle{every node}=[draw=black,fill=yellow!50!white,thick,
  shape=circle,minimum height=0.35cm,inner sep=2];


\def\phi{{asin(1/(2.0*sqrt(3)))}}
\def\rho{1.1}

\begin{scope}[xshift=-1.25cm]
  \begin{scope}[rotate=180]
    \begin{scope}[rotate=90]
      \draw (\phi:{\rho*sqrt(3)}) -- (-\phi:{\rho*sqrt(3)});  
    \end{scope}
    \foreach \th in {-\phi, \phi}
    {
      \begin{scope}[rotate=\th]
        \draw ( 0, 0) -- ( 60:\rho);
        \draw ( 0, 0) -- (120:\rho);
        \draw ( 0, {\rho*sqrt(3)}) -- ( 60:\rho);
        \draw ( 0, {\rho*sqrt(3)}) -- (120:\rho);
        \draw ( 60:\rho) -- (120: \rho);

        \node[fill=red] at ( 0, 0) {};
        \node[fill=blue]  at ( 60:\rho) {};
        \node[fill=green] at (120:\rho) {};
      \end{scope}
    }
    \begin{scope}[rotate=90]                  
      \node at (\phi:{\rho*sqrt(3)}) {};
      \node[fill=red] at (-\phi:{\rho*sqrt(3)}) {};
    \end{scope}

  \end{scope}
\end{scope}



\def\psi{{2*asin(sqrt(3)/6.0)}}    
\def\rot{{asin(sqrt(3)/6.0)+240}}

\begin{scope}[xshift=1.25cm,yshift=-0.9cm]   

  \foreach \th in {0, 60,...,300}
  {
    \begin{scope}[rotate=\th]
      \draw (0,0) -- (1,0) -- (0.5,{sqrt(3)/2.0});
    \end{scope}
  }
  \foreach \th in {0, 120, 240}
  {
    \begin{scope}[rotate=\th,shift={(-1,0)}]
      \draw (0,0) -- (\psi:1);
      \begin{scope}[shift={(\psi:1)}]
        \draw (0,0) -- (\rot:1);
      \end{scope}
    \end{scope}
  }
  \node[fill=red] at (0,0) {};
  \foreach \th in {0, 120, 240}
  {
    \begin{scope}[rotate=\th]
      \node[fill=blue] at (1,0) {};
      \node[fill=green] at (0.5,{sqrt(3)/2.0}) {};
    \end{scope}
  }
  \begin{scope}[shift={(-1,0)}]
    \node[fill=red] at (\psi:1) {};
  \end{scope}
  \begin{scope}[rotate=120,shift={(-1,0)}]
    \node[fill=blue] at (\psi:1) {};
  \end{scope}
  \begin{scope}[rotate=240,shift={(-1,0)}]
    \node at (\psi:1) {};
  \end{scope}

\end{scope}


\end{tikzpicture}
\caption{Unit-distance graphs requiring $4$ colors: the Moser spindle and the Golomb graph}
\label{mosergolomb}
\end{figure}

For more than half a century, the efforts to improve these bounds were unsuccessful. However, in a very recent breakthrough, Aubrey de Grey \cite{degrey} constructed a unit distance graph which cannot be $4$-colored thus improving the lower bound to $\chi(\mathbb{E}^2)\ge 5$. The purpose of this report is to provide a different proof of this result.

For the remainder of this paper, a \emph{proper $4$-coloring} of the plane, is a coloring using four colors, so that no two points at distance exactly $1$ from each other are identically colored . We plan to prove that no such coloring exists.

Our argument is based on a sequence of three assertions.

{\bf  (a)} \emph{Given any proper $4$-coloring of the plane, $\chi$, there exist two points $P$ and $Q$, at distance $\sqrt{11/3}$ from each other such that $\chi(P)=\chi(Q)$.}

{\bf  (b)} \emph{Let $\chi$ be a proper $4$-coloring of the plane, and let $P$ and $Q$ be two points distance $\sqrt{11/3}$ apart such that $\chi(P)=\chi(Q)$. Then, there exist three points $A, B$, and $C$, which are the vertices of an equilateral triangle with side length $1/\sqrt{3}$ such that $\chi(A)=\chi(B)=\chi(C)$.}

{\bf (c)} \emph{Let $A, B$ and $C$ be the vertices of an equilateral triangle of side $1/\sqrt{3}$. Then, there exists a unit distance graph containing $A$, $B$ and $C$ among its vertices, which cannot be $4$-colored under the restriction that $A$, $B$ and $C$ are identically colored.}

Thus, we will show that any proper $4$-coloring forces a monochromatic pair of points distance $\sqrt{11/3}$ apart, which in turn implies the existence of an monochromatic equilateral triangle of side length $1/\sqrt{3}$. However, under the assumption that a proper $4$-coloring exists, this latter statement contradicts assertion {\bf (c)}. It follows that no proper $4$-coloring of the plane exists, and therefore, $\chi(\mathbb{E}^2)\ge 5$.

The proofs of the above statements rely on the construction of several finite graphs that have the desired properties. Unsurprisingly, most of these graphs can be embedded
in $\mathbb{Q}[\sqrt{3},\sqrt{11}]\times \mathbb{Q}[\sqrt{3},\sqrt{11}] $. Note that both the Moser spindle and the Golomb graph share this property. De Grey used similar coordinates in his proof.
In fact, with one exception, all our graphs have embeddings with vertex coordinates of the form $(a\sqrt{3}/36+b\sqrt{11}/36, c/36+d\sqrt{3}\sqrt{11}/36)$, where $a, b, c, d$ are integers. In order to improve the formatting, throughout the remainder of the paper we will use the following notation:
\begin{equation}\label{notation}
[a,b,c,d]:=\left(\frac{a\sqrt{3}}{36}+\frac{b\sqrt{11}}{36}, \frac{c}{36}+\frac{d\sqrt{3}\sqrt{11}}{36}\right).
\end{equation}


All data files mentioned in this paper are available
at the url: \url{http://cs.indstate.edu/ge/ExooIsmailescuData}.

\section{\bf Proof of assertion (a)}

\begin{claim}\label{Claim1}
There exists a configuration of $79$ points in $\mathbb{Q}[\sqrt{3}, \sqrt{11},\sqrt{247}]\times \mathbb{Q}[\sqrt{3}, \sqrt{11},\sqrt{247}]$ such that for any proper $4$-coloring, there exist two points distance $\sqrt{11/3}$ apart which are identically colored.
\end{claim}
\begin{proof}
Consider first the following set of $40$ points in $\mathbb{Q}[\sqrt{3}, \sqrt{11}]\times \mathbb{Q}[\sqrt{3}, \sqrt{11}]$:
{\setstretch{1.0}
\begin{align*}
&{\bf [0, 0, 0, 0], [0, 0, 96, 0]}, [-33, -3, 33, -3], [-33, 3, 33, -9], [-33, 3, 33, 3], [-33, 3, 63, -3],\\
&[-33, 9, 63, 3],[-18, 0, 48, -6], [-18, 0, 48, 6], [-18, 6, 48, 0], [-15, -9, 15, 3], [-15, -3, 15, -3],\\
&[-15, -3, 45, 3], [-15, 3, -15, -3], [-15, 3, 15, -9], [-15, 3, 15, 3], [-15, 3, 81, -3], [-15, 3, 111, 3],\\
&[-15, 9, 15, -3], [-15, 9, 81, 3], [0, -12, 0, 0], [0, -6, 30, 0], [0, -6, 66, 0], [0, 0, 30, -6], [0, 0, 30, 6],\\
&[0, 0, 66, -6], [0, 0, 66, 6], [0, 6, 0, -6], [0, 6, 0, 6], [0, 6, 30, 0], [0, 6, 66, 0], [0, 6, 96, 6], [0, 12, 30, 6],\\
&[0, 12, 66, 6], [15, 3, 15, -3], [15, 3, 81, 3], [18, 0, 48, -6], [18, 6, 48, 0], [33, 3, 33, -3], [33, 3, 63, 3].
\end{align*}
}
The resulting graph, which we will denote $G_{40}$, has $82$ unit edges, and $59$ pairs at distance $\sqrt{11/3}$ from each other (henceforth referred to as $\sqrt{11/3}$ edges).

\begin{figure}[ht]
\centering
\begin{tikzpicture}[line width=1.5pt,scale=1.9]
\tikzstyle{every node}=[draw=black,fill=yellow!50!white,thick,
  shape=circle,minimum height=0.25cm,inner sep=2];
\coordinate (AA) at ( 3.21, 1.76);
\coordinate (AB) at ( 3.21, 6.35);
\coordinate (AC) at ( 0.00, 2.51);
\coordinate (AD) at ( 0.95, 0.86);
\coordinate (AE) at ( 0.95, 4.16);
\coordinate (AF) at ( 0.95, 3.95);
\coordinate (AG) at ( 1.90, 5.60);
\coordinate (AH) at ( 1.72, 2.40);
\coordinate (AI) at ( 1.72, 5.70);
\coordinate (AJ) at ( 2.67, 4.05);
\coordinate (AK) at ( 0.54, 3.30);
\coordinate (AL) at ( 1.49, 1.65);
\coordinate (AM) at ( 1.49, 4.73);
\coordinate (AN) at ( 2.44, 0.21);
\coordinate (AO) at ( 2.44, 0.00);
\coordinate (AP) at ( 2.44, 3.30);
\coordinate (AQ) at ( 2.44, 4.81);
\coordinate (AR) at ( 2.44, 7.89);
\coordinate (AS) at ( 3.40, 1.65);
\coordinate (AT) at ( 3.40, 6.46);
\coordinate (AU) at ( 1.31, 1.76);
\coordinate (AV) at ( 2.26, 3.19);
\coordinate (AW) at ( 2.26, 4.91);
\coordinate (AX) at ( 3.21, 1.54);
\coordinate (AY) at ( 3.21, 4.84);
\coordinate (AZ) at ( 3.21, 3.27);
\coordinate (BA) at ( 3.21, 6.56);
\coordinate (BB) at ( 4.16, 0.11);
\coordinate (BC) at ( 4.16, 3.41);
\coordinate (BD) at ( 4.16, 3.19);
\coordinate (BE) at ( 4.16, 4.91);
\coordinate (BF) at ( 4.16, 8.00);
\coordinate (BG) at ( 5.12, 4.84);
\coordinate (BH) at ( 5.12, 6.56);
\coordinate (BI) at ( 4.93, 1.65);
\coordinate (BJ) at ( 4.93, 6.46);
\coordinate (BK) at ( 4.70, 2.40);
\coordinate (BL) at ( 5.66, 4.05);
\coordinate (BM) at ( 6.42, 2.51);
\coordinate (BN) at ( 6.42, 5.60);
\draw[pink] (AA) -- (AC);
\draw[pink] (AA) -- (AE);
\draw[pink] (AA) -- (AW);
\draw[pink] (AA) -- (BE);
\draw[pink] (AA) -- (BM);
\draw[pink] (AB) -- (AF);
\draw[pink] (AB) -- (AV);
\draw[pink] (AB) -- (BD);
\draw[pink] (AB) -- (BN);
\draw[pink] (AC) -- (AW);
\draw[pink] (AC) -- (AZ);
\draw[pink] (AD) -- (AE);
\draw[pink] (AD) -- (AZ);
\draw[pink] (AD) -- (BB);
\draw[pink] (AE) -- (BA);
\draw[pink] (AE) -- (BC);
\draw[pink] (AE) -- (BE);
\draw[pink] (AF) -- (AX);
\draw[pink] (AF) -- (BD);
\draw[pink] (AG) -- (BD);
\draw[pink] (AG) -- (BF);
\draw[pink] (AG) -- (BG);
\draw[pink] (AH) -- (AI);
\draw[pink] (AH) -- (BI);
\draw[pink] (AI) -- (BJ);
\draw[pink] (AJ) -- (BI);
\draw[pink] (AJ) -- (BJ);
\draw[pink] (AK) -- (AS);
\draw[pink] (AL) -- (AQ);
\draw[pink] (AL) -- (BK);
\draw[pink] (AM) -- (AR);
\draw[pink] (AO) -- (AP);
\draw[pink] (AO) -- (BK);
\draw[pink] (AP) -- (AT);
\draw[pink] (AP) -- (BL);
\draw[pink] (AQ) -- (AS);
\draw[pink] (AQ) -- (BK);
\draw[pink] (AQ) -- (BL);
\draw[pink] (AS) -- (BL);
\draw[pink] (AT) -- (BL);
\draw[pink] (AU) -- (AW);
\draw[pink] (AU) -- (BB);
\draw[pink] (AU) -- (BC);
\draw[pink] (AV) -- (BG);
\draw[pink] (AW) -- (BH);
\draw[pink] (AX) -- (AY);
\draw[pink] (AY) -- (BF);
\draw[pink] (AY) -- (BN);
\draw[pink] (AZ) -- (BA);
\draw[pink] (AZ) -- (BB);
\draw[pink] (AZ) -- (BM);
\draw[pink] (BA) -- (BC);
\draw[pink] (BB) -- (BC);
\draw[pink] (BB) -- (BM);
\draw[pink] (BC) -- (BH);
\draw[pink] (BD) -- (BN);
\draw[pink] (BE) -- (BM);
\draw[pink] (BF) -- (BG);
\draw[pink] (BF) -- (BN);
\draw[blue] (AA) -- (AL);
\draw[blue] (AA) -- (AN);
\draw[blue] (AA) -- (AP);
\draw[blue] (AA) -- (AV);
\draw[blue] (AA) -- (BD);
\draw[blue] (AA) -- (BI);
\draw[blue] (AB) -- (AQ);
\draw[blue] (AB) -- (AR);
\draw[blue] (AB) -- (AW);
\draw[blue] (AB) -- (BE);
\draw[blue] (AB) -- (BJ);
\draw[blue] (AC) -- (AF);
\draw[blue] (AC) -- (AH);
\draw[blue] (AC) -- (AL);
\draw[blue] (AD) -- (AH);
\draw[blue] (AD) -- (AO);
\draw[blue] (AE) -- (AG);
\draw[blue] (AE) -- (AI);
\draw[blue] (AE) -- (AJ);
\draw[blue] (AE) -- (AP);
\draw[blue] (AF) -- (AH);
\draw[blue] (AF) -- (AJ);
\draw[blue] (AF) -- (AQ);
\draw[blue] (AG) -- (AJ);
\draw[blue] (AG) -- (AT);
\draw[blue] (AH) -- (AX);
\draw[blue] (AH) -- (AZ);
\draw[blue] (AI) -- (AY);
\draw[blue] (AI) -- (BA);
\draw[blue] (AJ) -- (BD);
\draw[blue] (AJ) -- (BE);
\draw[blue] (AK) -- (AM);
\draw[blue] (AK) -- (AU);
\draw[blue] (AK) -- (AV);
\draw[blue] (AL) -- (AN);
\draw[blue] (AL) -- (AV);
\draw[blue] (AL) -- (AX);
\draw[blue] (AM) -- (AP);
\draw[blue] (AM) -- (AV);
\draw[blue] (AM) -- (AY);
\draw[blue] (AN) -- (AS);
\draw[blue] (AN) -- (BB);
\draw[blue] (AO) -- (AX);
\draw[blue] (AO) -- (BB);
\draw[blue] (AP) -- (AY);
\draw[blue] (AP) -- (BC);
\draw[blue] (AP) -- (BD);
\draw[blue] (AQ) -- (AZ);
\draw[blue] (AQ) -- (BE);
\draw[blue] (AR) -- (AT);
\draw[blue] (AR) -- (BF);
\draw[blue] (AS) -- (BB);
\draw[blue] (AS) -- (BD);
\draw[blue] (AT) -- (BE);
\draw[blue] (AT) -- (BF);
\draw[blue] (AT) -- (BH);
\draw[blue] (AU) -- (AV);
\draw[blue] (AV) -- (AW);
\draw[blue] (AX) -- (AZ);
\draw[blue] (AX) -- (BB);
\draw[blue] (AX) -- (BI);
\draw[blue] (AX) -- (BK);
\draw[blue] (AY) -- (BA);
\draw[blue] (AY) -- (BC);
\draw[blue] (AZ) -- (BK);
\draw[blue] (BA) -- (BF);
\draw[blue] (BA) -- (BJ);
\draw[blue] (BB) -- (BI);
\draw[blue] (BC) -- (BG);
\draw[blue] (BD) -- (BE);
\draw[blue] (BD) -- (BI);
\draw[blue] (BD) -- (BL);
\draw[blue] (BE) -- (BJ);
\draw[blue] (BE) -- (BL);
\draw[blue] (BF) -- (BH);
\draw[blue] (BF) -- (BJ);
\draw[blue] (BG) -- (BH);
\draw[blue] (BI) -- (BM);
\draw[blue] (BJ) -- (BN);
\draw[blue] (BK) -- (BM);
\draw[blue] (BL) -- (BM);
\draw[blue] (BL) -- (BN);
\node[fill=red] at (AA) {};
\node[fill=red] at (AB) {};
\node at (AC) {};
\node at (AD) {};
\node at (AE) {};
\node at (AF) {};
\node at (AG) {};
\node at (AH) {};
\node at (AI) {};
\node at (AJ) {};
\node at (AK) {};
\node at (AL) {};
\node at (AM) {};
\node at (AN) {};
\node at (AO) {};
\node at (AP) {};
\node at (AQ) {};
\node at (AR) {};
\node at (AS) {};
\node at (AT) {};
\node at (AU) {};
\node at (AV) {};
\node at (AW) {};
\node at (AX) {};
\node at (AY) {};
\node at (AZ) {};
\node at (BA) {};
\node at (BB) {};
\node at (BC) {};
\node at (BD) {};
\node at (BE) {};
\node at (BF) {};
\node at (BG) {};
\node at (BH) {};
\node at (BI) {};
\node at (BJ) {};
\node at (BK) {};
\node at (BL) {};
\node at (BM) {};
\node at (BN) {};
\end{tikzpicture}
\caption{$G_{40}$ has $82$ unit edges (blue) and $59$ $\sqrt{11/3}$ edges (pink). The distance between then first two vertices $[0,0,0,0]$ and $[0,0,96,0]$ (shown in red) is $8/3$.}
\label{G40drawing}
\end{figure}
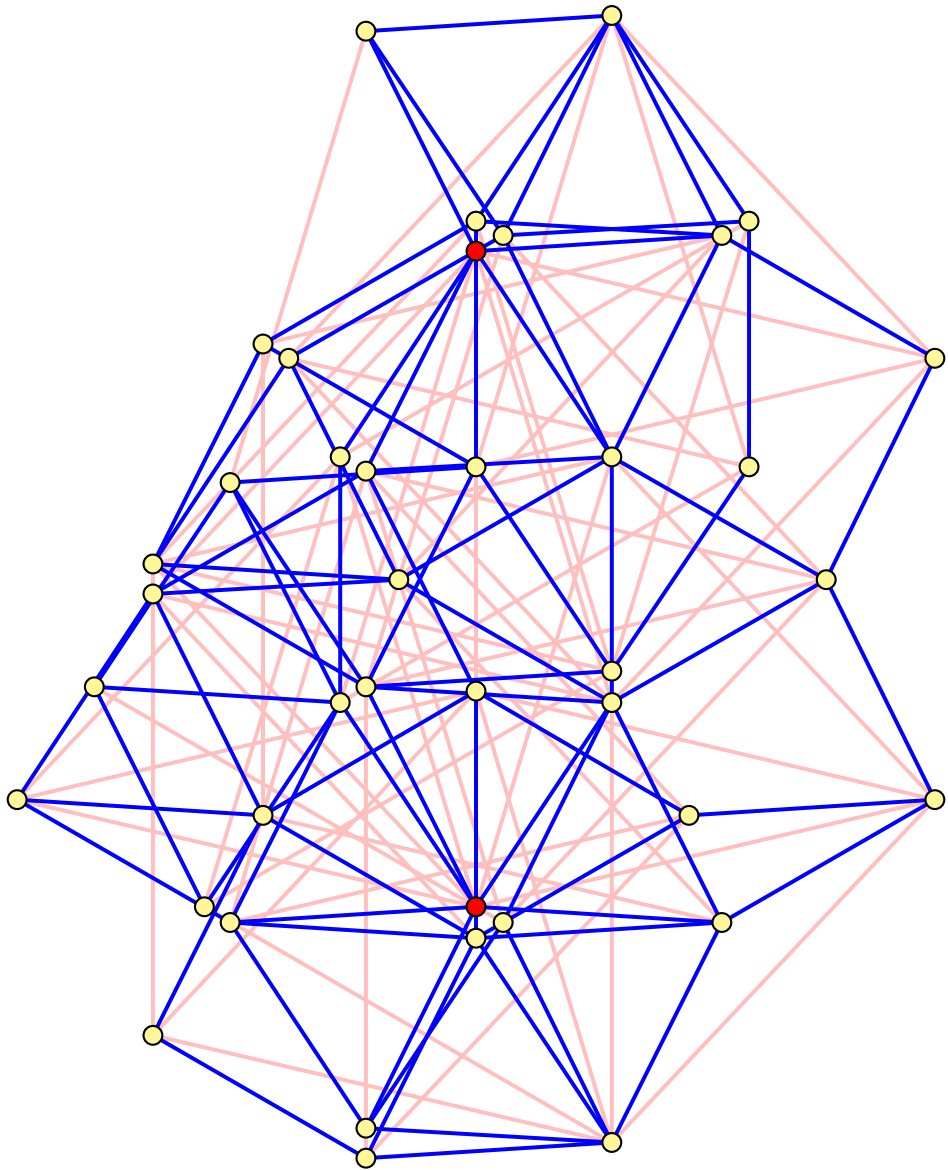

The crucial property of $G_{40}$ is that for any proper $4$-coloring which avoids monochromatic pairs distance $\sqrt{11/3}$ apart,  vertices $[0,0,0,0]$ and $[0,0,96,0]$ are colored the same. This can be easily verified by using Sage.

Next, rotate $G_{40}$ around vertex $[0,0,0,0]$ until the new position of $[0,0,96,0]$ is distance $1$ apart from its original location. Since the distance between vertices $[0,0,0,0]$ and $[0,0,96,0]$ equals $8/3$, it follows that the required rotation angle is $\theta= \arccos(119/128)=21^{\circ}.61\ldots$.

Consider now the graph $G_{79}$ whose vertex set consists of the union of the vertex set of $G_{40}$ and the vertex set of the rotated copy of $G_{40}$.

This graph has $2\cdot 82+1=165$ unit edges and $2\cdot 59= 118$ edges of length $\sqrt{11/3}$, and more importantly, any proper $4$-coloring of the vertices creates a monochromatic pair of points at distance $\sqrt{11/3}$ from each other.

Since $\sin{\theta}=3\sqrt{247}/128$ it follows that all the $79$ points lie in $\mathbb{Q}[\sqrt{3}, \sqrt{11},\sqrt{247}]\times \mathbb{Q}[\sqrt{3}, \sqrt{11},\sqrt{247}]$, as claimed.

\begin{figure}[ht]
\centering
\begin{tikzpicture}[line width=1.2pt,scale=1.9]
\tikzstyle{every node}=[draw=black,fill=yellow!50!white,thick,
  shape=circle,minimum height=0.2cm,inner sep=1];

\input{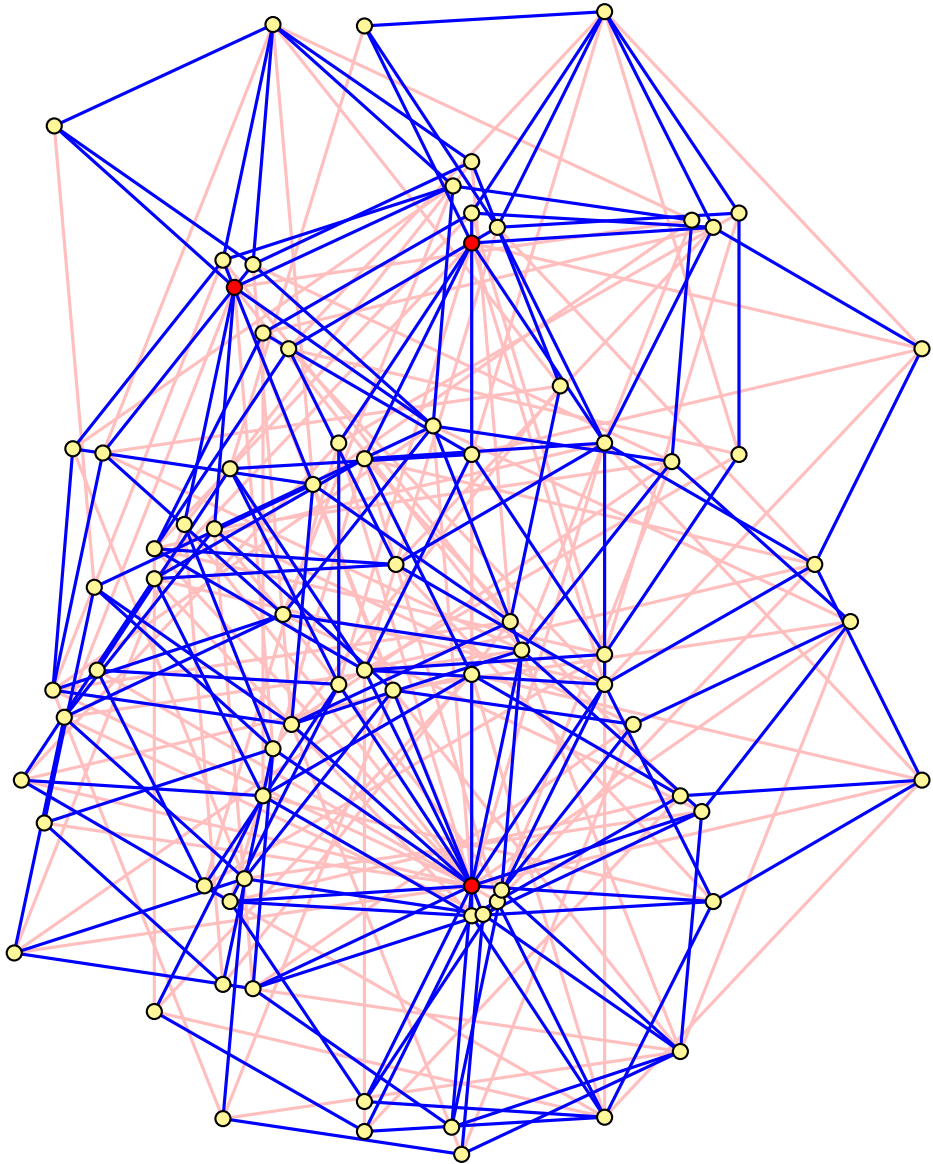}

\end{tikzpicture}

\caption{The graph $G_{79}$ obtained by rotating $G_{40}$ around the origin by the angle $\theta=\arccos(119/128)$. This graph has $165$ unit edges and $118$ $\sqrt{11/3}$ edges, and any proper $4$-coloring forces the endpoints of at least one of these $118$ edges to have the same color.}
\label{G79drawing}
\end{figure}
\end{proof}

\section{\bf Proof of assertion (b)}

\begin{claim}\label{Claim2}
There exists a configuration of $49$ points in $\mathbb{Q}[\sqrt{3}, \sqrt{11}]\times \mathbb{Q}[\sqrt{3}, \sqrt{11}]$ which contains two fixed points $P$ and $Q$ at distance $\sqrt{11/3}$ from each other, and which has following property: for any proper $4$-coloring, either points $P$ and $Q$ have different colors, or there exist three identically colored points $A, B$, and $C$ which are the vertices of an equilateral triangle with side length $1/\sqrt{3}$ .
\end{claim}

\begin{proof}
Consider the following set of $49$ points in $\mathbb{Q}[\sqrt{3}, \sqrt{11}]\times \mathbb{Q}[\sqrt{3}, \sqrt{11}]$ (see figure \ref{G49drawing} below):
{\setstretch{1.0}
\begin{align*}
&{\bf [0, 0, 0, 0], [0, 0, 0, 12]}, [-6, 0, 0, 6], [6, 0, 0, 6], [0, 0, -18, 6], [0, 0, 18, 6], [-3, -9, 9, 9], \\
&[3, 9, 9, 9], [-9, -9, 9, 3], [9, 9, 9, 3], [-3, -9, -9, 3], [3, 9, -9, 3], [-12, 0, 0, 0], [12, 0, 0, 0],\\
&[9, -9, -9, 3], [-9, 9, -9, 3], [3, -9, -9, 9], [-3, 9, -9, 9], [3, -9, 9, 3], [-3, 9, 9, 3], [-6, 0, 18, 0],\\
&[6, 0, 18, 0], [-6, 0, -18, 0], [6, 0, -18, 0], [0, -18, 0, 6], [0, 18, 0, 6], [-6, -18, 0, 0], [6, 18, 0, 0],\\
&[6, -18, 0, 0], [-6, 18, 0, 0], [-3, -9, 9, -3], [3, 9, 9, -3], [9, -9, 9, -3], [-9, 9, 9, -3], [-9, -9, -9, -3],\\
&[9, 9, -9, -3], [3, -9, -9, -3], [-3, 9, -9, -3], [-6, 0, 0, -6], [6, 0, 0, -6], [0, 0, 0, -12], [0, 0, -18, -6],\\
&[0, 0, 18, -6], [0, -18, 0, -6], [0, 18, 0, -6], [-3, -9, -9, -9], [-3, 9, 9, -9], [3, -9, 9, -9], [3, 9, -9, -9].
\end{align*}
}
\begin{figure}[ht]
\centering
\begin{tikzpicture}[line width=1.2pt,scale=1.3]
\tikzstyle{every node}=[draw=black,fill=yellow!50!white,thick,
  shape=circle,minimum height=0.2cm,inner sep=1];
\coordinate (AA) at ( 4.00, 3.93);
\coordinate (AB) at ( 4.00, 7.87);
\coordinate (AC) at ( 3.41, 5.90);
\coordinate (AD) at ( 4.59, 5.90);
\coordinate (AE) at ( 4.00, 4.87);
\coordinate (AF) at ( 4.00, 6.93);
\coordinate (AG) at ( 2.00, 7.40);
\coordinate (AH) at ( 6.00, 7.40);
\coordinate (AI) at ( 1.41, 5.43);
\coordinate (AJ) at ( 6.59, 5.43);
\coordinate (AK) at ( 2.00, 4.40);
\coordinate (AL) at ( 6.00, 4.40);
\coordinate (AM) at ( 2.81, 3.93);
\coordinate (AN) at ( 5.19, 3.93);
\coordinate (AO) at ( 3.19, 4.40);
\coordinate (AP) at ( 4.81, 4.40);
\coordinate (AQ) at ( 2.59, 6.37);
\coordinate (AR) at ( 5.41, 6.37);
\coordinate (AS) at ( 2.59, 5.43);
\coordinate (AT) at ( 5.41, 5.43);
\coordinate (AU) at ( 3.41, 4.96);
\coordinate (AV) at ( 4.59, 4.96);
\coordinate (AW) at ( 3.41, 2.91);
\coordinate (AX) at ( 4.59, 2.91);
\coordinate (AY) at ( 0.59, 5.90);
\coordinate (AZ) at ( 7.41, 5.90);
\coordinate (BA) at ( 0.00, 3.93);
\coordinate (BB) at ( 8.00, 3.93);
\coordinate (BC) at ( 1.19, 3.93);
\coordinate (BD) at ( 6.81, 3.93);
\coordinate (BE) at ( 2.00, 3.46);
\coordinate (BF) at ( 6.00, 3.46);
\coordinate (BG) at ( 3.19, 3.46);
\coordinate (BH) at ( 4.81, 3.46);
\coordinate (BI) at ( 1.41, 2.44);
\coordinate (BJ) at ( 6.59, 2.44);
\coordinate (BK) at ( 2.59, 2.44);
\coordinate (BL) at ( 5.41, 2.44);
\coordinate (BM) at ( 3.41, 1.97);
\coordinate (BN) at ( 4.59, 1.97);
\coordinate (BO) at ( 4.00, 0.00);
\coordinate (BP) at ( 4.00, 0.94);
\coordinate (BQ) at ( 4.00, 2.99);
\coordinate (BR) at ( 0.59, 1.97);
\coordinate (BS) at ( 7.41, 1.97);
\coordinate (BT) at ( 2.00, 0.47);
\coordinate (BU) at ( 5.41, 1.50);
\coordinate (BV) at ( 2.59, 1.50);
\coordinate (BW) at ( 6.00, 0.47);
\draw[blue] (AA) -- (AC);
\draw[blue] (AA) -- (AD);
\draw[blue] (AA) -- (AK);
\draw[blue] (AA) -- (AL);
\draw[blue] (AA) -- (AS);
\draw[blue] (AA) -- (AT);
\draw[blue] (AA) -- (BE);
\draw[blue] (AA) -- (BF);
\draw[blue] (AA) -- (BK);
\draw[blue] (AA) -- (BL);
\draw[blue] (AA) -- (BM);
\draw[blue] (AA) -- (BN);
\draw[blue] (AB) -- (AC);
\draw[blue] (AB) -- (AD);
\draw[blue] (AB) -- (AG);
\draw[blue] (AB) -- (AH);
\draw[blue] (AB) -- (AQ);
\draw[blue] (AB) -- (AR);
\draw[blue] (AC) -- (AG);
\draw[blue] (AC) -- (AI);
\draw[blue] (AC) -- (AK);
\draw[blue] (AC) -- (AM);
\draw[blue] (AC) -- (AP);
\draw[blue] (AC) -- (AR);
\draw[blue] (AC) -- (AT);
\draw[blue] (AD) -- (AH);
\draw[blue] (AD) -- (AJ);
\draw[blue] (AD) -- (AL);
\draw[blue] (AD) -- (AN);
\draw[blue] (AD) -- (AO);
\draw[blue] (AD) -- (AQ);
\draw[blue] (AD) -- (AS);
\draw[blue] (AE) -- (AF);
\draw[blue] (AE) -- (AK);
\draw[blue] (AE) -- (AL);
\draw[blue] (AE) -- (AQ);
\draw[blue] (AE) -- (AR);
\draw[blue] (AE) -- (AW);
\draw[blue] (AE) -- (AX);
\draw[blue] (AF) -- (AG);
\draw[blue] (AF) -- (AH);
\draw[blue] (AF) -- (AS);
\draw[blue] (AF) -- (AT);
\draw[blue] (AF) -- (AU);
\draw[blue] (AF) -- (AV);
\draw[blue] (AG) -- (AI);
\draw[blue] (AG) -- (AS);
\draw[blue] (AG) -- (AY);
\draw[blue] (AH) -- (AJ);
\draw[blue] (AH) -- (AT);
\draw[blue] (AH) -- (AZ);
\draw[blue] (AI) -- (AM);
\draw[blue] (AI) -- (AO);
\draw[blue] (AI) -- (AU);
\draw[blue] (AI) -- (BA);
\draw[blue] (AI) -- (BE);
\draw[blue] (AJ) -- (AN);
\draw[blue] (AJ) -- (AP);
\draw[blue] (AJ) -- (AV);
\draw[blue] (AJ) -- (BB);
\draw[blue] (AJ) -- (BF);
\draw[blue] (AK) -- (AQ);
\draw[blue] (AK) -- (AW);
\draw[blue] (AK) -- (AY);
\draw[blue] (AK) -- (BA);
\draw[blue] (AK) -- (BI);
\draw[blue] (AK) -- (BK);
\draw[blue] (AL) -- (AR);
\draw[blue] (AL) -- (AX);
\draw[blue] (AL) -- (AZ);
\draw[blue] (AL) -- (BB);
\draw[blue] (AL) -- (BJ);
\draw[blue] (AL) -- (BL);
\draw[blue] (AM) -- (AP);
\draw[blue] (AM) -- (AV);
\draw[blue] (AM) -- (AX);
\draw[blue] (AM) -- (BH);
\draw[blue] (AM) -- (BI);
\draw[blue] (AM) -- (BM);
\draw[blue] (AN) -- (AO);
\draw[blue] (AN) -- (AU);
\draw[blue] (AN) -- (AW);
\draw[blue] (AN) -- (BG);
\draw[blue] (AN) -- (BJ);
\draw[blue] (AN) -- (BN);
\draw[blue] (AO) -- (AQ);
\draw[blue] (AO) -- (AX);
\draw[blue] (AO) -- (BC);
\draw[blue] (AO) -- (BK);
\draw[blue] (AP) -- (AR);
\draw[blue] (AP) -- (AW);
\draw[blue] (AP) -- (BD);
\draw[blue] (AP) -- (BL);
\draw[blue] (AQ) -- (AY);
\draw[blue] (AR) -- (AZ);
\draw[blue] (AS) -- (AV);
\draw[blue] (AS) -- (AY);
\draw[blue] (AS) -- (BC);
\draw[blue] (AS) -- (BE);
\draw[blue] (AS) -- (BG);
\draw[blue] (AT) -- (AU);
\draw[blue] (AT) -- (AZ);
\draw[blue] (AT) -- (BD);
\draw[blue] (AT) -- (BF);
\draw[blue] (AT) -- (BH);
\draw[blue] (AU) -- (AW);
\draw[blue] (AU) -- (BE);
\draw[blue] (AU) -- (BH);
\draw[blue] (AU) -- (BQ);
\draw[blue] (AV) -- (AX);
\draw[blue] (AV) -- (BF);
\draw[blue] (AV) -- (BG);
\draw[blue] (AV) -- (BQ);
\draw[blue] (AW) -- (BI);
\draw[blue] (AW) -- (BL);
\draw[blue] (AW) -- (BP);
\draw[blue] (AX) -- (BJ);
\draw[blue] (AX) -- (BK);
\draw[blue] (AX) -- (BP);
\draw[blue] (AY) -- (BA);
\draw[blue] (AY) -- (BC);
\draw[blue] (AZ) -- (BB);
\draw[blue] (AZ) -- (BD);
\draw[blue] (BA) -- (BE);
\draw[blue] (BA) -- (BI);
\draw[blue] (BA) -- (BR);
\draw[blue] (BB) -- (BF);
\draw[blue] (BB) -- (BJ);
\draw[blue] (BB) -- (BS);
\draw[blue] (BC) -- (BG);
\draw[blue] (BC) -- (BK);
\draw[blue] (BC) -- (BR);
\draw[blue] (BD) -- (BH);
\draw[blue] (BD) -- (BL);
\draw[blue] (BD) -- (BS);
\draw[blue] (BE) -- (BM);
\draw[blue] (BE) -- (BQ);
\draw[blue] (BE) -- (BR);
\draw[blue] (BE) -- (BV);
\draw[blue] (BF) -- (BN);
\draw[blue] (BF) -- (BQ);
\draw[blue] (BF) -- (BS);
\draw[blue] (BF) -- (BU);
\draw[blue] (BG) -- (BI);
\draw[blue] (BG) -- (BN);
\draw[blue] (BG) -- (BV);
\draw[blue] (BH) -- (BJ);
\draw[blue] (BH) -- (BM);
\draw[blue] (BH) -- (BU);
\draw[blue] (BI) -- (BM);
\draw[blue] (BI) -- (BT);
\draw[blue] (BJ) -- (BN);
\draw[blue] (BJ) -- (BW);
\draw[blue] (BK) -- (BN);
\draw[blue] (BK) -- (BP);
\draw[blue] (BK) -- (BR);
\draw[blue] (BK) -- (BT);
\draw[blue] (BL) -- (BM);
\draw[blue] (BL) -- (BP);
\draw[blue] (BL) -- (BS);
\draw[blue] (BL) -- (BW);
\draw[blue] (BM) -- (BO);
\draw[blue] (BM) -- (BT);
\draw[blue] (BM) -- (BU);
\draw[blue] (BN) -- (BO);
\draw[blue] (BN) -- (BV);
\draw[blue] (BN) -- (BW);
\draw[blue] (BO) -- (BT);
\draw[blue] (BO) -- (BU);
\draw[blue] (BO) -- (BV);
\draw[blue] (BO) -- (BW);
\draw[blue] (BP) -- (BQ);
\draw[blue] (BP) -- (BT);
\draw[blue] (BP) -- (BW);
\draw[blue] (BQ) -- (BU);
\draw[blue] (BQ) -- (BV);
\draw[blue] (BR) -- (BT);
\draw[blue] (BR) -- (BV);
\draw[blue] (BS) -- (BU);
\draw[blue] (BS) -- (BW);
\node[fill=red] at (AA) {};
\node[fill=red] at (AB) {};
\node at (AC) {};
\node at (AD) {};
\node at (AE) {};
\node at (AF) {};
\node at (AG) {};
\node at (AH) {};
\node at (AI) {};
\node at (AJ) {};
\node at (AK) {};
\node at (AL) {};
\node at (AM) {};
\node at (AN) {};
\node at (AO) {};
\node at (AP) {};
\node at (AQ) {};
\node at (AR) {};
\node at (AS) {};
\node at (AT) {};
\node at (AU) {};
\node at (AV) {};
\node at (AW) {};
\node at (AX) {};
\node at (AY) {};
\node at (AZ) {};
\node at (BA) {};
\node at (BB) {};
\node at (BC) {};
\node at (BD) {};
\node at (BE) {};
\node at (BF) {};
\node at (BG) {};
\node at (BH) {};
\node at (BI) {};
\node at (BJ) {};
\node at (BK) {};
\node at (BL) {};
\node at (BM) {};
\node at (BN) {};
\node at (BO) {};
\node at (BP) {};
\node at (BQ) {};
\node at (BR) {};
\node at (BS) {};
\node at (BT) {};
\node at (BU) {};
\node at (BV) {};
\node at (BW) {};
\end{tikzpicture}
\vspace{2mm}
\caption{The graph $G_{49}$ has $180$ unit edges. Any proper $4$-coloring of the vertices of $G_{49}$ either assigns different colors to vertices $P=[0,0,0,0]$ and $Q=[0,0,0,12]$, or it forces a monochromatic equilateral triangle of side length $1/\sqrt{3}$.}
\label{G49drawing}
\end{figure}
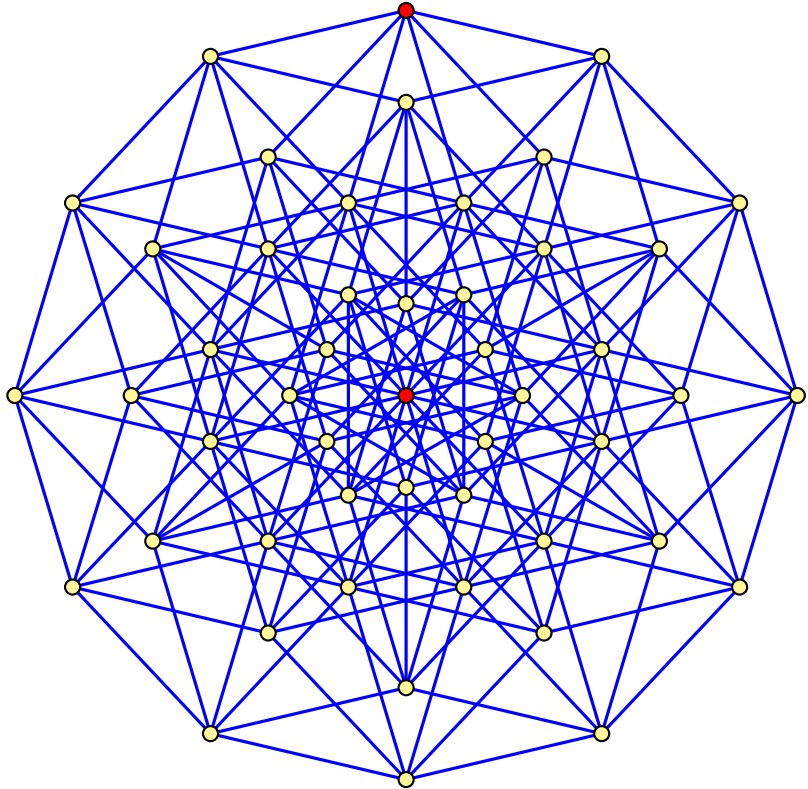

Let $G_{49}$ be the unit distance graph determined by these points. It can be easily checked that $G_{49}$ has $180$ unit edges and the distance between the first two vertices $P=[0,0,0,0]$ and $Q=[0,0,0,12]$ is $\sqrt{11/3}$. The graph $G_{49}$ contains $18$ equilateral triangles of side length $1/\sqrt{3}$ as shown in figure \ref{equilateral}.

\begin{figure}[ht]
\centering

\begin{tikzpicture}[line width=1.2pt,scale=1.5]
\tikzstyle{every node}=[draw=black,fill=yellow!50!white,thick,
  shape=circle,minimum height=0.15cm,inner sep=1];
\input{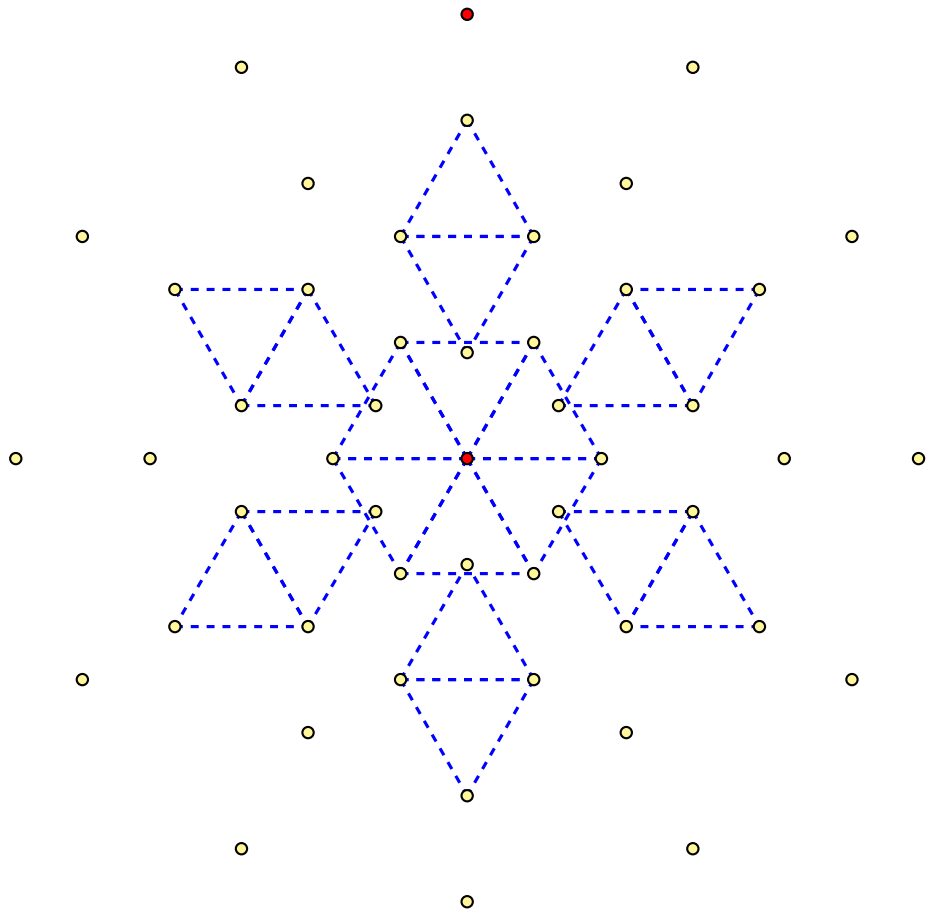}
\end{tikzpicture}

\caption{The $18$ equilateral triangles of side length $1/\sqrt{3}$ contained in $G_{49}$. }
\label{equilateral}
\end{figure}

It can be verified that $G_{49}$ has exactly $18694$ proper $4$-colorings up the color permutations, but only $44$ of these contain no monochromatic equilateral $1/\sqrt{3}$-triangle.
However, in each of these $44$ colorings, the first two vertices $P=[0,0,0,0]$ and $Q=[0,0,0,12]$ are colored differently. This concludes the proof.

\end{proof}

\section{\bf Proof of assertion (c).}

\begin{claim}\label{Claim3}
Let $A$, $B$ and $C$ be the vertices of an equilateral triangle of side length $1/\sqrt{3}$. There exists a unit distance graph of order $627$ containing $A$, $B$ and $C$ among its vertices which cannot be properly $4$-colored if points $A$, $B$ and $C$ are identically colored.
\end{claim}
We begin with a unit distance graph of order 51, denoted $G_{51}$, whose vertices are given below:
{\setstretch{1.0}
\begin{align*}
& {\bf [0, 0, 12, 0], [-6, 0, -6, 0], [6, 0, -6, 0]}, [0, 0, -24, 0], [0, 0, -6, -6], [0, 0, -6, 6], [-18, 0, -6, 0],\\
& [-12, 0, -24, 0], [-12, 0, 12, 0], [-12, 0, -6, -6], [-12, 0, -6, 6], [-6, 0, 30, 0], [-6, 0, 12, -6],\\
& [-6, 0, 12, 6], [18, 0, -6, 0], [12, 0, -24, 0], [12, 0, 12, 0], [12, 0, -6, -6], [12, 0, -6, 6],  [6, 0, 30, 0],\\
& [6, 0, 12, -6], [6, 0, 12, 6], [-9, -9, 3, -3], [-9, -9, -15, 3], [-9, 9, 3, 3], [-9, 9, -15, -3], \\
& [-3, -9, 3, 3], [-3, -9, 21, -3], [-3, -9, -15, -3], [-3, 9, 3, -3], [-3, 9, 21, 3], [-3, 9, -15, 3], \\
& [9, -9, 3, 3], [9, -9, -15, -3], [9, 9, 3, -3], [9, 9, -15, 3], [3, -9, 3, -3], [3, -9, 21, 3], [3, -9, -15, 3], \\
& [3, 9, 3, 3], [3, 9, 21, -3],  [3, 9, -15, -3], [0, 0, 30, 6], [-15, -9, -15, -3], [15, 9, -15, -3], \\
&[-6, 0, -24, 6], [6, 0, -24, 6], [-15, 9, 3, -3], [-9, 9, 21, -3], [9, -9, 21, -3], [15, -9, 3, -3].
\end{align*}
}

The graph $G_{51}$ has an automorphism group of order 6, generated by a $2 \pi/3$ rotation about the
origin and a reflection about the y-axis. The first three vertices of $G_{51}$ form an equilateral triangle
with side length $1/\sqrt{3}$.  Denote these vertices $A$, $B$, and $C$ - shown colored in red in Figure~\ref{fig:g51}.
Our goal is to show that none of the 4-colorings of $G_{51}$ in which these three vertices have the same color can be extended to the entire plane.

\begin{figure}[ht]
\centering

\begin{tikzpicture}[line width=1.2pt,scale=0.85]
\tikzstyle{every node}=[draw=black,fill=yellow!50!white,thick,
  shape=circle,minimum height=0.18cm,inner sep=1];
\coordinate (AA) at ( 7.00, 7.00);
\coordinate (AB) at ( 5.70, 4.60);
\coordinate (AC) at ( 8.30, 4.60);
\coordinate (AD) at ( 7.00, 2.20);
\coordinate (AE) at ( 7.00, 0.00);
\coordinate (AF) at ( 7.00, 9.20);
\coordinate (AG) at ( 3.09, 4.60);
\coordinate (AH) at ( 4.39, 2.20);
\coordinate (AI) at ( 4.39, 7.00);
\coordinate (AJ) at ( 4.39, 0.00);
\coordinate (AK) at ( 4.39, 9.20);
\coordinate (AL) at ( 5.70, 9.40);
\coordinate (AM) at ( 5.70, 2.40);
\coordinate (AN) at ( 5.70,11.60);
\coordinate (AO) at (10.91, 4.60);
\coordinate (AP) at ( 9.61, 2.20);
\coordinate (AQ) at ( 9.61, 7.00);
\coordinate (AR) at ( 9.61, 0.00);
\coordinate (AS) at ( 9.61, 9.20);
\coordinate (AT) at ( 8.30, 9.40);
\coordinate (AU) at ( 8.30, 2.40);
\coordinate (AV) at ( 8.30,11.60);
\coordinate (AW) at ( 1.30, 3.50);
\coordinate (AX) at ( 1.30, 5.70);
\coordinate (AY) at ( 8.79, 8.10);
\coordinate (AZ) at ( 8.79, 1.10);
\coordinate (BA) at ( 2.61, 8.10);
\coordinate (BB) at ( 2.61, 5.90);
\coordinate (BC) at ( 2.61, 1.10);
\coordinate (BD) at (10.09, 3.50);
\coordinate (BE) at (10.09,10.50);
\coordinate (BF) at (10.09, 5.70);
\coordinate (BG) at ( 5.21, 8.10);
\coordinate (BH) at ( 5.21, 1.10);
\coordinate (BI) at (12.70, 3.50);
\coordinate (BJ) at (12.70, 5.70);
\coordinate (BK) at ( 3.91, 3.50);
\coordinate (BL) at ( 3.91,10.50);
\coordinate (BM) at ( 3.91, 5.70);
\coordinate (BN) at (11.39, 8.10);
\coordinate (BO) at (11.39, 5.90);
\coordinate (BP) at (11.39, 1.10);
\coordinate (BQ) at ( 7.00,14.00);
\coordinate (BR) at ( 0.00, 1.10);
\coordinate (BS) at (14.00, 1.10);
\coordinate (BT) at ( 5.70, 6.80);
\coordinate (BU) at ( 8.30, 6.80);
\coordinate (BV) at ( 7.49, 3.50);
\coordinate (BW) at ( 8.79, 5.90);
\coordinate (BX) at ( 5.21, 5.90);
\coordinate (BY) at ( 6.51, 3.50);
\draw[blue] (AA) -- (AD);
\draw[blue] (AA) -- (AG);
\draw[blue] (AA) -- (AM);
\draw[blue] (AA) -- (AN);
\draw[blue] (AA) -- (AO);
\draw[blue] (AA) -- (AU);
\draw[blue] (AA) -- (AV);
\draw[blue] (AA) -- (BA);
\draw[blue] (AA) -- (BB);
\draw[blue] (AA) -- (BD);
\draw[blue] (AA) -- (BE);
\draw[blue] (AA) -- (BK);
\draw[blue] (AA) -- (BL);
\draw[blue] (AA) -- (BN);
\draw[blue] (AA) -- (BO);
\draw[blue] (AB) -- (AE);
\draw[blue] (AB) -- (AF);
\draw[blue] (AB) -- (AJ);
\draw[blue] (AB) -- (AK);
\draw[blue] (AB) -- (AL);
\draw[blue] (AB) -- (AP);
\draw[blue] (AB) -- (AQ);
\draw[blue] (AB) -- (AW);
\draw[blue] (AB) -- (AX);
\draw[blue] (AB) -- (AY);
\draw[blue] (AB) -- (AZ);
\draw[blue] (AB) -- (BA);
\draw[blue] (AB) -- (BC);
\draw[blue] (AB) -- (BD);
\draw[blue] (AB) -- (BF);
\draw[blue] (AC) -- (AE);
\draw[blue] (AC) -- (AF);
\draw[blue] (AC) -- (AH);
\draw[blue] (AC) -- (AI);
\draw[blue] (AC) -- (AR);
\draw[blue] (AC) -- (AS);
\draw[blue] (AC) -- (AT);
\draw[blue] (AC) -- (BG);
\draw[blue] (AC) -- (BH);
\draw[blue] (AC) -- (BI);
\draw[blue] (AC) -- (BJ);
\draw[blue] (AC) -- (BK);
\draw[blue] (AC) -- (BM);
\draw[blue] (AC) -- (BN);
\draw[blue] (AC) -- (BP);
\draw[blue] (AD) -- (AG);
\draw[blue] (AD) -- (AO);
\draw[blue] (AD) -- (BC);
\draw[blue] (AD) -- (BF);
\draw[blue] (AD) -- (BM);
\draw[blue] (AD) -- (BP);
\draw[blue] (AD) -- (BT);
\draw[blue] (AD) -- (BU);
\draw[blue] (AE) -- (BC);
\draw[blue] (AE) -- (BD);
\draw[blue] (AE) -- (BK);
\draw[blue] (AE) -- (BP);
\draw[blue] (AF) -- (BA);
\draw[blue] (AF) -- (BF);
\draw[blue] (AF) -- (BM);
\draw[blue] (AF) -- (BN);
\draw[blue] (AF) -- (BQ);
\draw[blue] (AG) -- (AJ);
\draw[blue] (AG) -- (AK);
\draw[blue] (AG) -- (BR);
\draw[blue] (AG) -- (BV);
\draw[blue] (AH) -- (AI);
\draw[blue] (AH) -- (AX);
\draw[blue] (AH) -- (AZ);
\draw[blue] (AH) -- (BR);
\draw[blue] (AH) -- (BT);
\draw[blue] (AI) -- (AM);
\draw[blue] (AI) -- (AN);
\draw[blue] (AI) -- (AT);
\draw[blue] (AI) -- (AW);
\draw[blue] (AI) -- (AY);
\draw[blue] (AI) -- (BV);
\draw[blue] (AI) -- (BW);
\draw[blue] (AJ) -- (AU);
\draw[blue] (AJ) -- (AW);
\draw[blue] (AJ) -- (AZ);
\draw[blue] (AJ) -- (BR);
\draw[blue] (AJ) -- (BV);
\draw[blue] (AK) -- (AV);
\draw[blue] (AK) -- (AX);
\draw[blue] (AK) -- (AY);
\draw[blue] (AK) -- (BU);
\draw[blue] (AL) -- (AQ);
\draw[blue] (AL) -- (BB);
\draw[blue] (AL) -- (BE);
\draw[blue] (AL) -- (BQ);
\draw[blue] (AL) -- (BW);
\draw[blue] (AM) -- (AR);
\draw[blue] (AM) -- (AW);
\draw[blue] (AM) -- (BB);
\draw[blue] (AM) -- (BD);
\draw[blue] (AM) -- (BW);
\draw[blue] (AN) -- (AS);
\draw[blue] (AN) -- (AY);
\draw[blue] (AN) -- (BA);
\draw[blue] (AN) -- (BE);
\draw[blue] (AN) -- (BT);
\draw[blue] (AO) -- (AR);
\draw[blue] (AO) -- (AS);
\draw[blue] (AO) -- (BS);
\draw[blue] (AO) -- (BY);
\draw[blue] (AP) -- (AQ);
\draw[blue] (AP) -- (BH);
\draw[blue] (AP) -- (BJ);
\draw[blue] (AP) -- (BS);
\draw[blue] (AP) -- (BU);
\draw[blue] (AQ) -- (AU);
\draw[blue] (AQ) -- (AV);
\draw[blue] (AQ) -- (BG);
\draw[blue] (AQ) -- (BI);
\draw[blue] (AQ) -- (BX);
\draw[blue] (AQ) -- (BY);
\draw[blue] (AR) -- (BH);
\draw[blue] (AR) -- (BI);
\draw[blue] (AR) -- (BS);
\draw[blue] (AR) -- (BY);
\draw[blue] (AS) -- (BG);
\draw[blue] (AS) -- (BJ);
\draw[blue] (AS) -- (BT);
\draw[blue] (AT) -- (BL);
\draw[blue] (AT) -- (BO);
\draw[blue] (AT) -- (BQ);
\draw[blue] (AT) -- (BX);
\draw[blue] (AU) -- (BI);
\draw[blue] (AU) -- (BK);
\draw[blue] (AU) -- (BO);
\draw[blue] (AU) -- (BX);
\draw[blue] (AV) -- (BG);
\draw[blue] (AV) -- (BL);
\draw[blue] (AV) -- (BN);
\draw[blue] (AV) -- (BU);
\draw[blue] (AW) -- (BA);
\draw[blue] (AW) -- (BH);
\draw[blue] (AW) -- (BX);
\draw[blue] (AX) -- (BC);
\draw[blue] (AX) -- (BG);
\draw[blue] (AX) -- (BR);
\draw[blue] (AX) -- (BT);
\draw[blue] (AY) -- (BD);
\draw[blue] (AY) -- (BJ);
\draw[blue] (AY) -- (BV);
\draw[blue] (AZ) -- (BF);
\draw[blue] (AZ) -- (BI);
\draw[blue] (AZ) -- (BW);
\draw[blue] (BA) -- (BK);
\draw[blue] (BB) -- (BC);
\draw[blue] (BB) -- (BL);
\draw[blue] (BB) -- (BY);
\draw[blue] (BC) -- (BM);
\draw[blue] (BC) -- (BY);
\draw[blue] (BD) -- (BN);
\draw[blue] (BD) -- (BS);
\draw[blue] (BE) -- (BF);
\draw[blue] (BE) -- (BO);
\draw[blue] (BE) -- (BQ);
\draw[blue] (BE) -- (BW);
\draw[blue] (BF) -- (BP);
\draw[blue] (BF) -- (BT);
\draw[blue] (BG) -- (BK);
\draw[blue] (BG) -- (BY);
\draw[blue] (BH) -- (BM);
\draw[blue] (BH) -- (BX);
\draw[blue] (BI) -- (BN);
\draw[blue] (BI) -- (BW);
\draw[blue] (BJ) -- (BP);
\draw[blue] (BJ) -- (BS);
\draw[blue] (BJ) -- (BU);
\draw[blue] (BK) -- (BR);
\draw[blue] (BL) -- (BM);
\draw[blue] (BL) -- (BQ);
\draw[blue] (BL) -- (BX);
\draw[blue] (BM) -- (BU);
\draw[blue] (BO) -- (BP);
\draw[blue] (BO) -- (BV);
\draw[blue] (BP) -- (BV);
\node[fill=red] at (AA) {};
\node[fill=red] at (AB) {};
\node[fill=red] at (AC) {};
\node at (AD) {};
\node at (AE) {};
\node at (AF) {};
\node at (AG) {};
\node at (AH) {};
\node at (AI) {};
\node at (AJ) {};
\node at (AK) {};
\node at (AL) {};
\node at (AM) {};
\node at (AN) {};
\node at (AO) {};
\node at (AP) {};
\node at (AQ) {};
\node at (AR) {};
\node at (AS) {};
\node at (AT) {};
\node at (AU) {};
\node at (AV) {};
\node at (AW) {};
\node at (AX) {};
\node at (AY) {};
\node at (AZ) {};
\node at (BA) {};
\node at (BB) {};
\node at (BC) {};
\node at (BD) {};
\node at (BE) {};
\node at (BF) {};
\node at (BG) {};
\node at (BH) {};
\node at (BI) {};
\node at (BJ) {};
\node at (BK) {};
\node at (BL) {};
\node at (BM) {};
\node at (BN) {};
\node at (BO) {};
\node at (BP) {};
\node at (BQ) {};
\node at (BR) {};
\node at (BS) {};
\node at (BT) {};
\node at (BU) {};
\node at (BV) {};
\node at (BW) {};
\node at (BX) {};
\node at (BY) {};
\end{tikzpicture}
\caption{$G_{51}$ with vertices $A$, $B$ and $C$ in red.}
\label{fig:g51}
\end{figure}
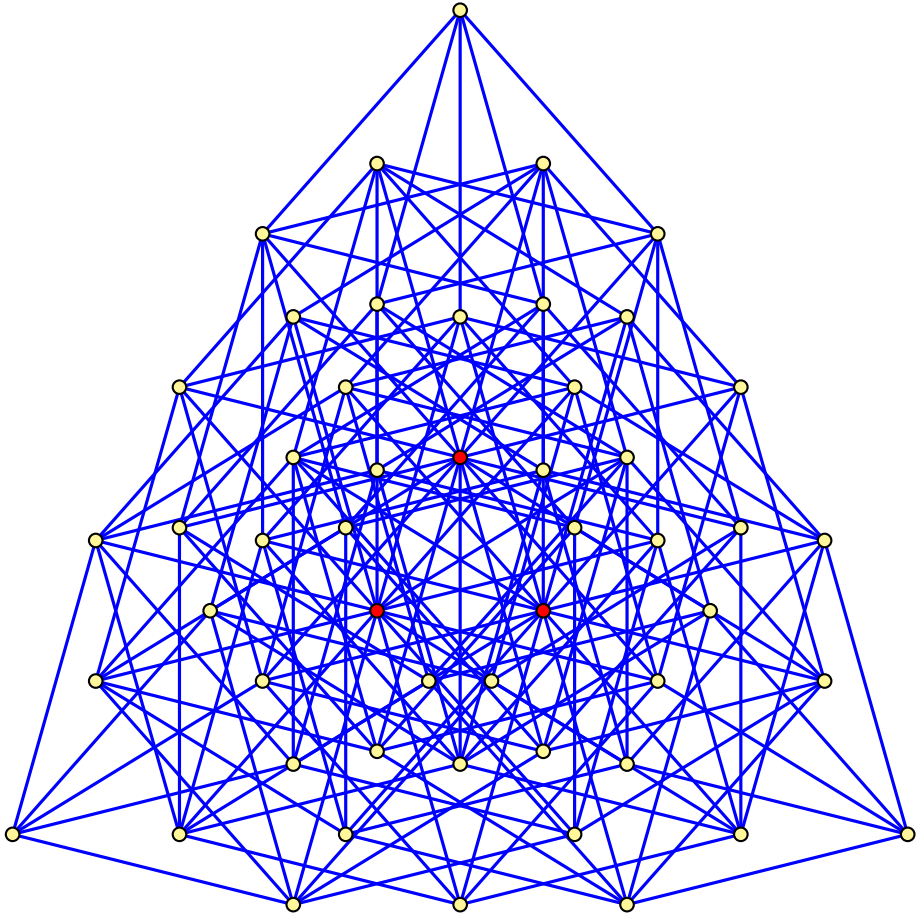

There are $13357$ proper $4$-colorings of $G_{51}$ for which $\chi(A)=\chi(B)=\chi(C)$, a fact that Sage, Maple or Mathematica can quickly verify.

Our method considers each of these colorings separately. For each coloring of  $G_{51}$, repeatedly add vertices whose colors are
forced, until arriving at a conflict. The added vertices are chosen by examining all triples of vertices in which each vertex of the triple has a different color,
computing the circumradius determined by the three vertices, and adding the circumcenter to the graph whenever the circumradius equals $1$.
For each coloring, we are eventually led to a color conflict, that is, we obtain two points distance $1$ apart that are identically colored.

For most of the 13357 colorings, this procedure produces such a conflict very quickly.
More than half of the cases were settled after at most 3 new vertices were added.

The most challenging coloring is shown below. The numbering is consistent with the ordering of the vertices of $G_{51}$ given earlier.
\begin{align*}
&Color_1=\{1, 2, 3, 43, 44, 45, 46, 49, 51\},\\
&Color_2=\{4, 5, 6, 8, 11, 14, 18, 20, 21, 23, 28, 33, 36, 48\},\\
&Color_3=\{7, 13, 15, 17, 25, 26, 29, 31, 34, 37, 38, 40, 42, 47\},\\
&Color_4=\{9, 10, 12, 16, 19, 22, 24, 27, 30, 32, 35, 39, 41, 50\}.
\end{align*}

This particular coloring of $G_{51}$ required 55 additional vertices until a conflict was reached. We give these vertices below, in the order they were added
to the graph. They are also available at \cite{website}.

{\setstretch{1.0}
\begin{align*}
&[6, 0, -24, -6], [-15, -9, 3, 3], [9, 9, 21, 3], [12, 0, 30, -6], [-21, -9, -15, 3], [-15, -9, 21, -3], [15, 9, 21, -3], \\
&[9, 9, 39, -3], [3, 9, 39, 3], [-6, -6, 0, 0], [3, 3, 9, 3], [-9, 9, 39, 3], [-15, -9, -33, 3], [24, 0, -6, -6],\\
&[-21, -9, 3, -3], [-18, 0, 30, 0], [18, 0, 30, 0], [-15, -9, 39, 3], [9, -3, 15, -3], [-3, -3, 15, -3], [9, 3, 9, -3],\\
&[-18, -18, 12, 0], [-21, -3, 15, 3], [6, -6, 0, 0], [15, -9, 39, -3], [-18, -18, -24, 0], [-3, 3, 9, -3], \\
&[-18, -12, 6, 0], [12, 6, 6, 0], [24, 6, 24, -6], [-9, 3, 9, 3], [9, -3, -21, -3], [-3, 3, 45, -3], [12, -6, 18, 0],\\
&[-6, 6, 24, 0], [-12, -6, 0, 6], [12, 6, -30, 0], [12, 6, 42, 0], [3, 3, 45, 3], [0, -12, 24, 0], [27, -3, 33, -3],\\
&[4, 0, -6, -2], [0, -6, 54, 0], [-18, -12, 42, 0], [30, 6, 24, 0], [6, -6, 18, 6], [-10, -6, 18, 2], [-2, 0, -6, 4],\\
&[6, -6, 54, 6],[30, 6, -12, 0], [-10, -6, 54, 2], [8, -6, 0, 2], [22, 0, 12, -2], [8, -6, 36, 2], [16, 0, 12, 4].
\end{align*}
}

While some colorings require more additional points than others until a conflict is generated,
one can find a set of 576 additional vertices that will eliminate all of the 13357 colorings of $G_{51}$.

This gives a graph of order 51+576=627, denoted $G_{627}$, that cannot be $4$-colored when $A$, $B$ and $C$
are colored the same. This graph has $2982$ unit edges, and as $G_{51}$, has an automorphism group of order 6
generated by a $2 \pi/3$ rotation about the origin and a reflection about the y-axis. The vertex coordinates of $G_{627}$ are available at the url \cite{website}.
A drawing of this graph is shown in the figure below.

\vspace{4mm}

\begin{figure}[ht]
\centering

\begin{tikzpicture}[line width=0.5pt,scale=1.2]
\tikzstyle{every node}=[draw=black,fill=yellow!50!white,thick,
  shape=circle,minimum height=0.10cm,inner sep=1];
\input{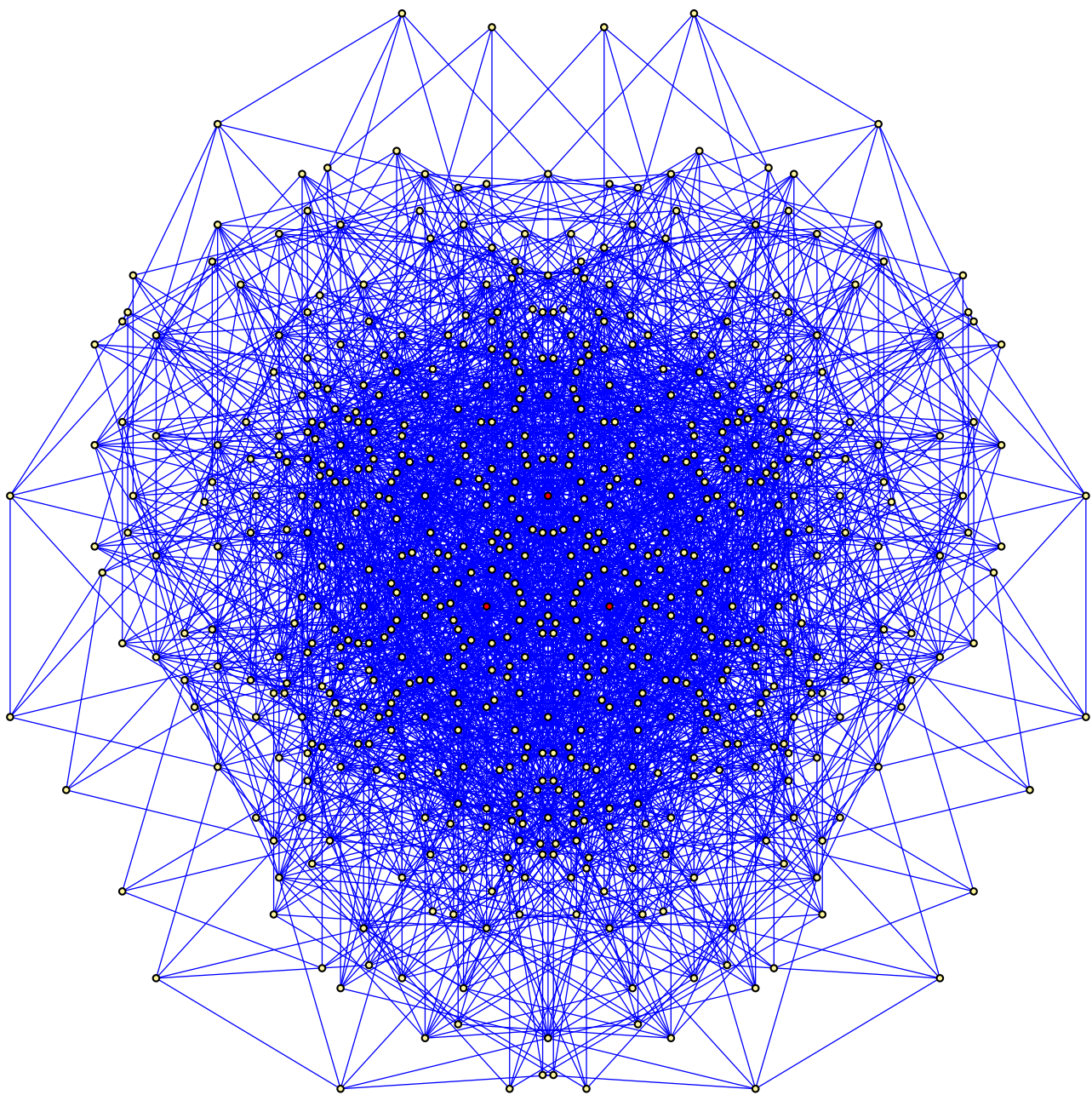}
\end{tikzpicture}
\caption{The graph $G_{627}$ is not 4-colorable if the first three vertices $A$, $B$ and $C$ are identically colored}
\label{fig:g627}
\end{figure}

\vspace{4mm}

For completeness, in the appendix we include a list of 109 vertices that can be used to generate the graph $G_{627}$ by rotating
these points around the origin by multiples of $2\pi/3$ and/or reflecting across the lines $y=0$ and $\sqrt{3}x\pm y =0$.
The full set of coordinates of the vertices of the graph $G_{627}$ is available at the url \cite{website}.

To verify that $G_{627}$ is not $4$-colorable under the restriction $\chi(A)=\chi(B)=\chi(C)$  we can use a simple
depth-first-search
coloring procedure.
The algorithm maintains a list of colors, $available$
that may be used at each vertex.  Initially all colors are available at each
vertex.  When a vertex is colored, the available colors for its neighbors may be
updated.
Note that it is important that the vertices be processed in the order given in our data file.
In that case all colors for vertices 52 to 627 are forced.
The verification takes a small fraction of a second when implemented in ${\bf C}$.  Even a Python
implementation takes only a few seconds.

\vspace{5mm}

\begin{algorithm}[htb!]
\caption{Four Coloring Search}\label{algo:fourcolor}
\begin{algorithmic}[1]
\Procedure{ColorVertex}{Vertex v, Color c}
   \State $color[v] \gets c$
   \ForAll {$u \textit{ adjacent to } v$}
      \State Remove $c$ from $available[u]$.
      \If {$available[u]$ is empty}
         \State Return False
      \ElsIf {$available[u]$ contains only one color $c'$}
         \If {Not ColorVertex($u$,$c'$)}
            \State Return False
         \EndIf
      \EndIf
   \EndFor
   \State Return True
\EndProcedure

\Statex

\Procedure{DFS}{Vertex $v$}
   \If {$v > $n}
      \State SaveColoring()
      \State Exit()
   \EndIf
   \For {$c \in available[v]$}
       \State Save state.
       \If {ColorVertex(v, c)}
          \State DFS(v+1)
       \EndIf
       \State Restore state.
   \EndFor
\EndProcedure

\Statex

\Procedure{FourColor}{Graph $G$}
   \State Initialize $available$
   \For {$v \gets 1,2,3$}
      \State ColorVertex(v, 1)
   \EndFor
   \State DFS(4)
   \State Output("No coloring.")
\EndProcedure
\end{algorithmic}
\end{algorithm}

\section{\bf Comments}

Using the previous three claims, one can construct a $5$-chromatic unit distance graph as follows.
Start with a copy of $G_{79}$. For each of the 118 edges of length $\sqrt{11/3}$ place a copy of $G_{49}$ so that the
marked vertices of $G_{49}$ coincide with the endpoints of that edge. We obtain a graph of order $118\cdot 49=5782$ where some
vertices are counted multiple times. Then, Claim \ref{Claim1} and Claim \ref{Claim2} imply that any proper $4$-coloring of this graph
forces at least one of the $118\cdot 18=2124$ equilateral triangles of side length $1/\sqrt{3}$ contained in the copies of $G_{49}$ to be
monochromatic. Finally, for each such triangle construct a copy of $G_{627}$ so that vertices $A$, $B$ and $C$ of $G_{627}$ coincide
with the vertices of the triangle.  The resulting unit distance graph cannot be properly colored with 4 colors.
The order of this graph is much larger than 20425, the order of the first de Grey's graph, hence determining the exact answer is not all that interesting.

The reason for this discrepancy is that de Grey's approach is more economical as it requires only two steps. De Grey first constructs
a graph of order 121 which contains 52 equilateral triangles of side length $\sqrt{3}$, and which has the property that for any proper $4$-coloring
at least one of these 52 triangles is monochromatic. Second, de Grey shows that such a monochromatic triple can be blocked by a graph of order 1345.

De Grey's 20425-vertex graph can be embedded in $\mathbb{Q}[\sqrt{3},\sqrt{11},\sqrt{5},\sqrt{7}]\times \mathbb{Q}[\sqrt{3},\sqrt{11},\sqrt{5},\sqrt{7}]$.
In comparison, our method is less efficient as it employs three steps; our $5$-chromatic graph can be embedded in
$\mathbb{Q}[\sqrt{3}, \sqrt{11},\sqrt{247}]\times \mathbb{Q}[\sqrt{3}, \sqrt{11},\sqrt{247}]$.

Very recently, Marijn Heule \cite{mixon} found a 826-vertex $5$-chromatic unit distance graph which can be embedded in $\mathbb{Q}[\sqrt{3},\sqrt{11},\sqrt{5}]\times \mathbb{Q}[\sqrt{3},\sqrt{11},\sqrt{5}]$.

One interesting question is whether there exists a $5$-chromatic unit distance graph which can be embedded in
$\mathbb{Q}[\sqrt{3},\sqrt{11}]\times \mathbb{Q}[\sqrt{3},\sqrt{11}]$. If true, this would be the best possible as a recent result of Madore shows \cite{madore}.
More generally, what is the smallest subfield of $\mathbb{E}^2$ that requires 5 colors?

There are many more interesting questions that are wide open. What is the minimum order of a $5$-chromatic unit distance graph in the plane? Is it possible to
use the newly discovered ideas to construct an unit distance graph which requires six colors?

Can one prove that the \emph{fractional chromatic number} of the plane is greater than 4? Cranston and Rabern \cite{CR} constructed a sequence of unit-distance graphs
whose fractional chromatic number approaches $105/29=3.6206\ldots$. This was improved in \cite{EI} by the authors of the present report  who found a unit-distance graph of order
73 whose fractional chromatic number is $383/102=3.7549\ldots$.

It would also be interesting to investigate if any of these techniques can be used to improve the lower bounds for the chromatic number of higher dimensional Euclidean spaces.

\section{\bf Appendix}
We include a list of 109 vertices that can be used to generate the graph $G_{627}$ in Claim \ref{Claim3}.
The full set of 627 vertices is obtained by rotating these points around the origin by multiples of $2\pi/3$
and/or reflecting across the lines $y=0$ and $\sqrt{3}x\pm y =0$.

{\setstretch{1.0}
\begin{align*}
&[0, 0, 12, 0], [0, 0, -24, 0], [0, 0, -6, -6], [0, 0, -6, 6], [-18, 0, -6, 0], [-12, 0, -6, -6], [-12, 0, -6, 6],\\
&[-6, 0, 12, -6], [-6, 0, 12, 6], [0, 0, 30, 6], [-6, 0, -24, 6], [-32, -6, -12, 2], [-32, 6, -12, -2], \\
&[-30, -6, -18, -6], [-30, -6, -12, 0], [-30, -6, 6, 6], [-30, -6, 24, 0], [-30, 0, 12, 6], [-30, 6, -18, 6], \\
&[-30, 6, -12, 0], [-30, 6, 6, -6], [-30, 6, 24, 0], [-27, -9, -15, 9], [-27, -9, 3, 3], [-27, -9, 21, -3], \\
&[-27, -3, -27, 3], [-27, -3, -3, 3], [-27, 3, -3, -3], [-27, 9, -15, 3], [-27, 9, 3, -3], [-27, 9, 21, 3], \\
&[-24, -6, -18, 0], [-24, -6, 24, -6], [-24, 0, -24, 0], [-24, 0, -6, -6], [-24, 0, -6, 6], [-24, 0, 12, 0], \\
&[-24, 6, -18, 0], [-24, 18, -6, 0], [-22, -6, -12, 4], [-22, 0, -6, -4], [-22, 0, -6, 4], [-22, 0, 12, -2],\\
&[-22, 0, 12, 2], [-21, -9, -15, -9], [-21, -9, -15, 3], [-21, -9, 3, -3], [-21, -9, 21, 3], [-21, -3, -3, -3],\\
&[-21, -3, 15, 3], [-21, 3, -9, -3], [-21, 3, -3, 3], [-21, 3, 15, -3], [-21, 9, -15, -3], [-21, 9, -15, 9],\\
&[-21, 9, 3, -9], [-21, 9, 3, 3], [-21, 9, 21, -3], [-19, -3, -3, -1], [-19, 3, -3, 1], [-18, -18, 12, 0], \\
&[-18, -12, 6, 0], [-18, 0, 12, -6], [-18, 0, 12, 6], [-18, 12, 6, 0], [-18, 18, -6, 6], [-18, 18, 12, 0],\\
&[-16, 0, 12, -4], [-16, 0, 12, 4], [-15, -9, 3, 3], [-15, -3, -3, 3], [-15, -3, 9, 3], [-15, 3, -3, -3], \\
&[-15, 3, 9, -3], [-15, 9, -15, 3], [-15, 9, 3, 9], [-14, -6, 6, 2], [-14, 0, 12, 2], [-14, 6, 6, -2],\\
&[-13, -3, -3, 5], [-12, -6, -12, 6], [-12, -6, 0, -6], [-12, -6, 0, 6], [-12, -6, 6, 0], [-12, 6, 0, -6],\\
&[-12, 6, 0, 6], [-12, 6, 6, 0], [-12, 12, 0, 0], [-12, 18, -6, 0], [-9, -9, 3, 9], [-9, -3, -9, 3], \\
&[-9, -3, -3, -3], [-9, -3, 9, -3], [-9, 3, -3, 3], [-9, 9, 3, -9], [-9, 15, -3, 3], [-8, -6, 0, -2], \\
&[-8, 6, 0, 2], [-6, -12, 6, 0], [-6, -6, 0, 0], [-6, -6, 6, -6], [-6, 0, -6, 12], [-6, 6, 0, 0],\\
&[-5, -3, -3, 1], [-5, 3, -3, -1], [-4, -6, 0, 2], [-4, 6, 0, -2], [-3, -3, -3, 3], [-3, 3, -3, 9].
\end{align*}
}

\end{document}